\documentclass[a4paper,12pt]{article}
\usepackage{times, url}
\usepackage[utf8]{inputenc}
\usepackage[T1]{fontenc}
\usepackage{etoolbox}
\usepackage{amsmath,amsthm,amsfonts,latexsym,amssymb,mathrsfs,graphics,pst-all,bbm,amscd,xypic,verbatim,epsfig,pstricks,makeidx}
\usepackage{amsthm, enumitem}
\usepackage{mathtools}

\newtheorem{theorem}{Theorem}[section]
\newtheorem{lemma}[theorem]{Lemma}

\newtheorem{corollary}[theorem]{Corollary}
\newtheorem{example}{Example}[section]
\newtheorem{remark}[theorem]{Remark}

\usepackage[center]{caption}
\usepackage[none]{hyphenat}
\usepackage{datetime}

\newdateformat{monthdayyeardate}{%
  \monthname[\THEMONTH]~\THEDAY, \THEYEAR}

\usepackage{color, colortbl}
\definecolor{Gray}{gray}{0.9}
\usepackage{amsthm, enumitem}

\makeatother

\begin{document}
\setcounter{page}{1}

\begin{center}
{\LARGE \bf  On General Sombor index}
\vspace{8mm}

{\bf Phanjoubam Chinglensana$^1$, Sainkupar Mn Mawiong$^2$ }
\vspace{3mm}

$^1$Department of Mathematics, North-Eastern Hill University,\\  
NEHU Campus, Shillong-793022, INDIA\\
e-mail: \url{phanjoubam17@gmail.com}
\vspace{2mm}

$^2$ Department of Basic Sciences and Social Sciences, North-Eastern Hill University,\\
  NEHU Campus, Shillong-793022, INDIA\\
e-mail: \url{skupar@gmail.com}
\vspace{2mm}

\monthdayyeardate\today

\end{center}
\vspace{10mm}

{\bf Abstract:} We present the bounds in terms of other important graph parameters for general Sombor index which generalises both the forgotten index and the Sombor index. We also explore the Nordhaus-Gaddum-type result for the general Sombor index. We present further the relations between general Sombor index and other generalised indices: general Randi\'c index and general sum-connectivity index.
\vspace{5mm}

\section{Introduction}
We consider only finite simple graph in this paper. Let $G$ be a finite simple graph on $n$ vertices and $m$ edges. We denote the vertex set and the edge set of $G$ by $V(G)$ and $E(G)$, respectively. The degree of a vertex $u\in V(G)$ is denoted by $d_G(u)$ and it is defined as the number of edges incident with $u$. The complement of $G$, denoted by $\overline{G}$, is a simple graph on $V(G)$ in which two vertices $u$ and $v$ are adjacent, i.e., joined by an edge $uv$, if and only if they are not adjacent in $G$. Hence, $uv\in E(\overline{G})$ if and only if $uv\notin E(G)$. Clearly, $E (G)\cup E(\overline{G}) = E(K_n )$, where $K_n$ is the complete graph on $n$ vertices and $\overline{m} =|E(\overline{G})| = {n\choose2}- m.$ Let $\Delta$ and $\delta$ denote the maximum vertex degree and the minimum vertex degree of the graph $G$, respectively.

In chemical graph theory, one generally considers various graph-theoretical invariants of molecular graphs (also known as topological indices or molecular descriptors), and study how strongly are they correlated with various properties of the corresponding molecules. The first such topological index was introduced in 1947 by H. Wiener \cite{18} and is used for correlation with boiling points of alkanes. Wiener's index is related to the distances in molecular graphs. Historically, the first vertex-degree-based topological indices were the graph invariants that nowadays are called Zagreb indices. Numerous graph invariants have been (and still continues to be) employed with varying degrees of success in QSAR (quantitative structure-activity relationship) and QSPR (quantitative structure-property relationship) studies. 
The Zagreb indices are amongst the most studied invariants \cite{14} and they are defined as sums of contributions dependent on the degrees of adjacent vertices over all edges of a graph. The Zagreb indices of a graph $G$, i.e., the {\it{first Zagreb index}} $M_1(G)$ and the {\it{second Zagreb index}} $M_2(G)$, were originally defined \cite{11} as follows.\\
$$M_1(G)=\sum_{u\in V(G)} {d_G(u)}^2; \ \ M_2(G)=\sum_{uv\in E(G)} d_G(u)d_G(v).$$
The first Zagreb index of $G$ can also be expressed as $$M_1(G)=\sum_{uv\in E(G)} [d_G(u)+d_G(v)].$$ Generalised version of the first Zagreb index have also been introduced \cite{13}, known as the {\it{general first Zagreb index}} and is defined as 
$$M_1^p (G)=\sum_{u\in V(G)} {d_G(u)}^p.$$ When $p=3$, $M_1^3 (G) =\displaystyle\sum_{u\in V(G)} {d_G(u)}^3$ is known as the {\it{forgotten index}}, denoted by $F(G)$, and is also equal to 
$$F(G)=\sum_{uv\in E(G)}[ {d_G(u)}^2+ {d_G(v)}^2].$$
One of the highly successful and widely used indices in QSPR and QSAR is the {\it{Randi\'c index}} $R(G)$ \cite{16}. It is defined as $$R(G)=\sum_{uv\in E(G)}\dfrac{1}{\sqrt{d_G(u)d_G(v)}}.$$ 
 The Randi\'c index have been extended to {\it general Randi\'c index} \cite{11}, defined as
 $$R_\alpha(G)=\sum_{uv\in E(G)}[d_G(u)d_G(v)]^\alpha$$ for any real number $\alpha$.
Motivated by Randi\'c and Zagreb indices, Zhou and Trinajsti\'c defined {\it sum-connectivity index} $\chi(G)$ \cite{19} and {\it general sum-connectivity index} $\chi_\alpha(G)$ \cite{20}, which are defined as
$$\chi(G)=\sum_{uv\in E(G)}\dfrac{1}{\sqrt{d_G(u)+d_G(v)}}$$ and
$$\chi_\alpha(G)=\sum_{uv\in E(G)}[d_G(u)+d_G(v)]^\alpha$$ for any real number $\alpha$. Gutman recently introduced \cite9 a new vertex-degree-based topological index called the {\it{Sombor index}}, which is defined as $SO(G)=\displaystyle\sum_{uv\in E(G)} \sqrt{{d_G(u)}^2+{d_G(v)}^2}$. Numerous work have so far been carried out on the Sombor index. The chemical applicability of the Sombor index is found to have shown good predictive potential \cite{15}. Computations of Sombor index of various graphs have been carried out, for example chemical graphs \cite5. Basic properties of the Sombor index have been presented and its relations with other topological indices: the Zagreb indices, are investigated in \cite6. In \cite3 and \cite{17}, in addition to Zagreb indices, relations between Sombor index and other topological indices are carried out.

Motivated by the extensions of Randi\'c and sum-connectivity indices and several works on Sombor index, we define {\it general Sombor index} $SO_\alpha(G)$. It is defined as $$SO_\alpha(G)=\displaystyle\sum_{uv\in E(G)} [{d_G(u)}^2+{d_G(v)}^2]^{\alpha/2}$$ for any real number $\alpha$.
For $\alpha=1$, we have the usual Sombor index whereas for $\alpha=2$, we get the forgotten index. Thus $SO_\alpha(G)$ generalises the Sombor index and the forgotten index. In this paper, we present the bounds of general Sombor index in terms of other important graph parameters.  We also explore the Nordhaus-Gaddum-type result for the general Sombor index. We present further the relations between general Sombor index and other generalised indices: general Randi\'c index and general sum-connectivity index.
\section{Preliminaries}
A graph $G$ is called {\it regular} if all vertices of $G$ have the same vertex degree and it is called {\it bi-degreed} if it has two distinct vertex degrees. A connected graph $G$ is called a {\it bi-regular graph} or {\it semi-regular bipartite graph} if $G$ is a bipartite graph with two partite sets $A$ and $B$ such that each vertex in $A$ has degree $\Delta$ and each vertex in $B$ has degree $\delta.$

Now, we state a lemma which provides necessary and sufficient conditions for non-regular graphs to be bi-regular \cite2.
\begin{lemma} \label{lemma1} Let $G$ be a connected non-regular graph. Then the following statements are equivalent:
\begin{enumerate}
\item $G$ is bi-regular.
\item $G$ is bi-degreed and $|d_G(u) - d_G(v)| > 0$ is constant for all edges $uv$ of $G$.
\item $d_G(u) + d_G(v) > 0$ is constant for all edges $uv$ of $G$.
\end{enumerate}
\end{lemma}
\begin{remark} \label{remark1} We note that the third condition in Lemma \ref{lemma1} can be replaced by ${d_G(u)}^2 + {d_G(v)}^2 > 0$ is constant for all edges $uv$ of $G$. Thus the following statements are equivalent for a connected non-regular graph $G$:
\begin{enumerate}
\item $G$ is bi-regular.
\item $G$ is bi-degreed and $|d_G(u)- d_G(v)| > 0$ is constant for all edges $uv$ of $G$.
\item ${d_G(u)}^2 + {d_G(v)}^2> 0$ is constant for all edges $uv$ of $G$. 
\end{enumerate}
\end{remark}
\begin{proof} [Proof of the Remark \ref{remark1}] Notice that (2) implies (3) is clear from the definition. So, we only prove that (3) implies (1). Now by (3) we have that ${d_G(u)}^2 + {d_G(v)}^2 > 0$ is constant for all edges $uv$ of $G$. This implies that any two vertices joined by a path of even length in $G$ must have the same degree. We now show that $G$ is bipartite by contradiction. Suppose that $G$ is not bipartite. Then $G$ has an odd cycle $C$. Let $u,v$ be two adjacent vertices on $C$. Then $C$ contains a path of even length connecting $u$ and $v$. Hence, $d_G(u) = d_G(v) = k$ for some positive integer $k$. Since $G$ is connected, for any vertex $x$ of $G$, there is a path $P$ between $u$ and $x$. Let $e=ab$ be any edge on the path $P$. Then by (3), $d_G(a)^2+d_G(b)^2=2k^2$ since ${d_G(u)}^2 + {d_G(v)}^2 =2k^2$ for the edge $uv$ of $G$. Also $d_G(u) = k.$ Then each vertex on the path $P$ has vertex degree $k$ and so does $x$. It follows that $G$ is regular. This contradicts the assumption that $G$ is not regular. Hence, $G$ is bipartite. Furthermore, any two vertices $u, v$ in the same partite set of $G$ are joined by a path of even length, and so $d_G(u) = d_G(v).$ Thus $G$ is bi-regular. This shows that (3) implies (1).
\end{proof}
Next, we recall the famous Jensen's inequality (see \cite4). 
\begin{lemma}[Jensen's inequality] Let $f:(a,b)\to \mathbb{R}$ be a convex function. Let $n\in \mathbb{N}$ and $\alpha_1,\alpha_2,\dots, \alpha_n\in (0,1)$ be real numbers such that $\alpha_1+ \alpha_2+\dots+\alpha_n=1$. Then for any $x_1,x_2,\dots,x_n\in (a,b)$ we have $$f\left(\sum_{i=1}^n \alpha_i x_i\right) \leq \sum_{i=1}^n \alpha_i f(x_i).$$
\end{lemma}
\begin{corollary} \label{corollary1} For a positive integer $k$, if $f$ is strictly convex (i.e., $f''>0$), then $$f\left(\sum_{i=1}^k \dfrac{ x_i}{k}\right) \leq \dfrac{1}{k} \sum_{i=1}^k f(x_i)$$ where the equality holds if and only if $x_1=x_2=\dots=x_k.$ Moreover, the inequality is reversed if $-f$ is strictly convex.
\end{corollary}
The following inequality is due to Radon and can be found in  \cite6.
\begin{lemma}[Radon's inequality] \label{lemma2}
If $a_k, b_k>0$ for $k=1,2,\dots, m$ and $p>0$, then $$\sum_{k=1}^m \dfrac{{a_k}^{p+1}}{{b_k}^p}\geq \dfrac{\left(\displaystyle\sum_{k=1}^m a_k\right)^{p+1}}{\left(\displaystyle\sum_{k=1}^m b_k\right)^p}.$$ Equality holds if $\dfrac{a_1}{b_1}=\dfrac{a_2}{b_2}=\dots=\dfrac{a_m}{b_m}$.
\end{lemma}
\section{General Sombor index and its properties}
In this section, we present the bounds of general Sombor index in terms of other important graph parameters.  We also explore the Nordhaus-Gaddum-type result for the general Sombor index. First, we give some examples of general Sombor index.
\begin{example}
\begin{enumerate}
\item $SO_\alpha(K_n)= 2^{-1+\alpha/2}n(n-1)^{\alpha+1}=2^{\alpha/2}m(n-1)^\alpha$ and $SO_\alpha(\overline{K_n})=0.$
\item $SO_\alpha(C_n)=2^{(3\alpha)/2}n.$
\item $SO_\alpha(P_2)=SO_\alpha(K_2)=2^{\alpha/2}.$
\item $SO_\alpha(P_n)=2\times5^{\alpha/2}+2(n-3)2^{\alpha/2}$ for $n\geq3$.
\end{enumerate}
\end{example}
\subsection{Bounds for the general Sombor index}
We first present bounds for the general Sombor index in terms of the forgotten index and numbers of edges of the graph.
\begin{theorem} Let $G$ be a graph with $m\geq1$ edges. Then we have the following.
\begin{enumerate} 
\item $SO_\alpha(G)\geq m^{1-\alpha/2} F(G)^{\alpha/2}$ \  if \ \  $\alpha<0$ or \ \ $\alpha>1$ and
\item $SO_\alpha(G)\leq m^{1-\alpha/2} F(G)^{\alpha/2}$ \  if \ \  $0<\alpha<1$.
\end{enumerate}
Moreover, equality holds in either cases if and only if $G$ is bi-regular.
\end{theorem}
\begin{proof}Notice that if $\alpha<0$ or $\alpha>1$, then $x^{\alpha/2} $ for $x>0$ is strictly convex. Thus by Corollary \ref{corollary1}, we have
\begin{align*}
\left[\dfrac{1}{m} F(G)\right]^{\alpha/2}=&\left[\dfrac{1}{m} \displaystyle\sum_{uv\in E(G)} ({d_G(u)}^2+{d_G(v)}^2)\right]^{\alpha/2}\\
=&\left[ \displaystyle\sum_{uv\in E(G)} \dfrac{{d_G(u)}^2+{d_G(v)}^2}{m}\right]^{\alpha/2}\\
\leq& \dfrac{1}{m} \displaystyle\sum_{uv\in E(G)} [{d_G(u)}^2+{d_G(v)}^2]^{\alpha/2}=\dfrac{1}{m} SO_\alpha(G)
\end{align*}
Thus $SO_\alpha(G)\geq m^{1-\alpha/2} F(G)^{\alpha/2}$ and the equality holds if and only if ${d_G(u)}^2 + {d_G(v)}^2$ is a constant for all edges $uv$ of $G$. That is, by Remark \ref{remark1} the equality holds if and only if $G$ is bi-regular. 
Similarly,  if $0<\alpha<1$, then $-x^{\alpha/2} $ for $x>0$ is strictly convex. Hence $SO_\alpha(G)\leq m^{1-\alpha/2} F(G)^{\alpha/2}$, where the equality holds if and only if $G$ is bi-regular.
\end{proof}
It is reported in \cite8 that $F(G)\geq \dfrac{M_1(G)^2}{2m}$ for a graph with $m$ edges, where the equality holds if and only if $G$ is a regular graph. Also, $M_1(G)\geq \dfrac{4m^2}{n} $, where the equality holds if and only if $G$ is a regular graph (due to the Cauchy-Schwarz inequality) \cite1. Thus we have the following bounds for the general Sombor index in terms of $n$ and $m$.
\begin{corollary} Let $G$ be a graph on $n$ vertices and $m\geq1$ edges. Then
\begin{enumerate} 
\item $SO_\alpha(G)\geq 8^{\alpha/2} m^{1+\alpha} n^{-\alpha}$ \  if \ \  $\alpha<0$ or $\alpha>1$ and
\item $SO_\alpha(G)\leq 8^{\alpha/2} m^{1+\alpha}n^{-\alpha}$ \  if \ \  $0<\alpha<1$.
\end{enumerate}
The equality holds in either cases if and only if $G$ is regular.
\end{corollary}
Next, we present another bound for the general Sombor index.
\begin{theorem} \label{theorem2} Let $G$ be a graph on $n\geq2$ vertices and $m$ edges. Then we have the following.
\begin{enumerate}
\item If $0<\alpha<1$, then $$SO_\alpha(G)\geq F(G)^{\alpha/2}$$ where the equality holds if and only if $G=K_2\cup \overline{K_{n-2}}$ or $G=\overline{K_n}.$
\item If $\alpha<0$, then $$SO_\alpha(G)\leq 2^{-1+\alpha/2}n(n-1)$$ where the equality holds if and only if $G=K_2$.
\item If $\alpha>1$, then $$SO_\alpha(G)\leq 2^{\alpha/2}m(n-1)^\alpha$$  where the equality holds if and only if $G=K_n$ or $G=\overline{K_n}.$
\end{enumerate}
\end{theorem}
\begin{proof} Let $0<\alpha<1$. Then $$SO_\alpha(G)\geq \left[\displaystyle\sum_{uv\in E(G)} [{d_G(u)}^2+{d_G(v)}^2]\right]^{\alpha/2}=F(G)^{\alpha/2}$$ where the equality holds if and only if $|E(G)|\leq1$, i.e., $G=K_2\cup \overline{K_{n-2}}$ or $G=\overline{K_n}.$ This proves (1). 

Let $\alpha<0$ and $G\neq \overline{K_n}$. Let $G_1$ be the graph obtained from $G$ by deleting the possible isolated vertices of $G$. Let $n_1=|V(G_1)|$ and let $\Delta_1$ and $\delta_1$ denote the maximum and minimum degrees of $G_1$, respectively. Notice that by Handshake lemma, we have $m_1\leq \dfrac{n_1\Delta_1}{2}$. Then
$$SO_\alpha(G)\leq \displaystyle\sum_{uv\in E(G_1)} (2\delta_1^2)^{\alpha/2}\leq \dfrac{n_1\Delta_1}{2} (2\delta_1^2)^{\alpha/2}=2^{-1+\alpha/2}n_1 \Delta_1 {\delta_1}^\alpha\leq 2^{-1+\alpha/2}n(n-1),$$
where the equality holds if and only if $G$ is regular and $\Delta_1 \delta_1^\alpha=n-1$, i.e., $G=K_2$. This proves (2). 

Let $\alpha>1$. Notice that $d_G(u)\leq n-1$ for any vertex $u$ of $G$. Thus 
$$SO_\alpha(G)=\displaystyle\sum_{uv\in E(G)} [{d_G(u)}^2+{d_G(v)}^2]^{\alpha/2}\leq m [2(n-1)^2]^{\alpha/2}=2^{\alpha/2}m(n-1)^\alpha $$ where the equality holds if and only if either $d_G(u)=d_G(v)=n-1$ for every edge $uv$ of $G$ or $E(G)=\emptyset$, i.e., $G=K_n$ or $G=\overline{K_n}$. This completes the proof.
\end{proof}
\subsection{Nordhaus-Gaddum-type result for the general Sombor index}
\begin{theorem} Let $G$ be a graph on $n\geq2$ vertices and $m$ edges.
\begin{enumerate}
\item If $\alpha>0$, then
$$SO_\alpha(G)+SO_\alpha(\overline{G})\leq 2^{-1+\alpha/2}n(n-1)^{\alpha+1}$$ where the equality holds if and only if $G=K_n$ or $G=\overline{K_n}$.
Further, if $\alpha\geq 1$, then
$$SO_\alpha(G)+SO_\alpha(\overline{G})\geq  \dfrac{n(n-1)^{1+\alpha}}{2^{1+\alpha}}$$ and
if $0<\alpha<1$, then
$$SO_\alpha(G)+SO_\alpha(\overline{G}) \geq  \dfrac{n^{\alpha/2}(n-1)^{(3\alpha)/2}}{2^{(3\alpha)/2}}.$$
\item If $\alpha<0$, then
$$2^{-1+\alpha/2}n(n-1)^{\alpha+1}\leq SO_\alpha(G)+SO_\alpha(\overline{G}) < 2^{\alpha/2} n(n-1)$$ where the left equality holds if and only if $G=K_n$ or $G=\overline{K_n}$.
\end{enumerate}
\end{theorem}
\begin{proof} Let $\alpha>0$. Notice that $d_G(u)\leq n-1$ for any vertex $u$ of $G$. Thus by definition, we have
\begin{align*}
SO_\alpha(G)+SO_\alpha(\overline{G}) =&\displaystyle\sum_{uv\in E(G)} [{d_G(u)}^2+{d_G(v)}^2]^{\alpha/2}+\displaystyle\sum_{uv\in E(\overline{G})} [{d_{\overline{G}}(u)}^2+{d_{\overline{G}}(v)}^2]^{\alpha/2} \\
\leq& m[2(n-1)^2]^{\alpha/2} + \overline{m}[2(n-1)^2]^{\alpha/2}\\
=&2^{\alpha/2}(m+\overline{m})(n-1)^{\alpha} = 2^{-1+\alpha/2}n(n-1)^{\alpha+1}
\end{align*}
where the equality holds if and only if either $d_G(u)=d_G(v)=n-1$ for every edge $uv$ of $G$ or $E(G)=\emptyset$, i.e., $G=K_n$ or $G=\overline{K_n}$. This proves the first part of (1). Similarly, if $\alpha<0$, then $SO_\alpha(G)+SO_\alpha(\overline{G})\geq 2^{-1+\alpha/2}n(n-1)^{\alpha+1}$, where the equality holds if and only if $G=K_n$ or $G=\overline{K_n}$. Further, by Theorem \ref{theorem2} (2), we have 
$$SO_\alpha(G)+SO_\alpha(\overline{G})<2^{-1+\alpha/2} n(n-1) +2^{-1+\alpha/2} n(n-1)=2^{\alpha/2} n(n-1).$$ This proves (2).\\
Notice that $m+\overline{m}=\dfrac{n(n-1)}{2}$, thus $F(G)\geq \dfrac{ M_1(G)^2}{2m}\geq \dfrac{ M_1(G)^2}{n(n-1)}$ and $F(\overline{G})\geq \dfrac{ M_1(\overline{G})^2}{2\overline{m}}\geq \dfrac{ M_1(G)^2}{n(n-1)}$. So, for $\alpha=2$, 
\begin{align}
SO_\alpha(G)+SO_\alpha(\overline{G})=&F(G)+F(\overline{G})\notag \\
\geq &  \dfrac{1}{n(n-1)}\left[M_1(G)^2+M_1(\overline{G})^2\right]\notag \\
\geq & \dfrac{1}{n(n-1)}\dfrac{\left[M_1(G)+M_1(\overline{G})\right]^2}{2}\notag \\
\geq &\dfrac{1}{n(n-1)}\dfrac{1}{2}\dfrac{n^2(n-1)^4}{4}=\dfrac{n(n-1)^3}{2^3} \label{eqn:eq1}
\end{align}
Now if $\alpha>1$, then $x^{\alpha/2}$ is strictly convex. Thus
\begin{align}
&SO_\alpha(G)+SO_\alpha(\overline{G})\notag\\
&\geq (m+\overline{m})\left[\dfrac{\displaystyle\sum_{uv\in E(G)} [{d_G(u)}^2+{d_G(v)}^2]+\displaystyle\sum_{uv\in E(\overline{G})} [{d_{\overline{G}}(u)}^2+{d_{\overline{G}}(v)}^2]}{m+\overline{m}}\right]^{\alpha/2}\notag \\
= & (m+\overline{m})^{1-\alpha/2}[F(G)+F(\overline{G})]^{\alpha/2} \label{eqn:eq2}
\end{align}
From (\ref{eqn:eq1}) and (\ref{eqn:eq2}), for $\alpha>1$ we have 
$$SO_\alpha(G)+SO_\alpha(\overline{G})\geq \dfrac{n(n-1)^{1+\alpha}}{2^{1+\alpha}}.$$
Lastly, if $0<\alpha<1$, then 
\begin{align}
SO_\alpha(G)+SO_\alpha(\overline{G})\geq &\left[\displaystyle\sum_{uv\in E(G)} [{d_G(u)}^2+{d_G(v)}^2]+\displaystyle\sum_{uv\in E(\overline{G})} [{d_{\overline{G}}(u)}^2+{d_{\overline{G}}(v)}^2]\right]^{\alpha/2}\notag \\
= & [F(G)+F(\overline{G})]^{\alpha/2} \label{eqn:eq3}
\end{align}
From (\ref{eqn:eq1}) and (\ref{eqn:eq3}), for $0<\alpha<1$ we have
$$SO_\alpha(G)+SO_\alpha(\overline{G})\geq \dfrac{n^{\alpha/2}(n-1)^{(3\alpha)/2}}{2^{(3\alpha)/2}}.$$
This completes the proof.
\end{proof}
\section{Comparison of general Sombor index with other generalised indices} 
In this section, we present the relations between general Sombor index and other generalised indices: general Randi\'c index and general sum-connectivity index.
\subsection{General Sombor index and general Randi\'c index}
We now present a relation between general Sombor index and Randi\'c index.
\begin{theorem} Let $G$ be a graph on $n$ vertices and $m$ edges. Then $$ \dfrac{2^{\alpha/2}}{\Delta^\alpha}R_\alpha(G)\leq SO_\alpha(G) \leq \dfrac{2^{\alpha/2}}{\delta^\alpha}R_\alpha(G)$$ where the equality holds if and only if $G$ is regular.
\end{theorem}
\begin{proof} Notice that 
\begin{align*}
SO_\alpha(G)=&\displaystyle\sum_{uv\in E(G)} [{d_G(u)}^2+{d_G(v)}^2]^{\alpha/2}\\
=&\displaystyle\sum_{uv\in E(G)} [d_G(u)d_G(v)]^\alpha\left[\dfrac{1}{d_G(u)^2}+\dfrac{1}{d_G(v)^2}\right]^{\alpha/2}\\
\geq&\displaystyle\sum_{uv\in E(G)} [d_G(u)d_G(v)]^\alpha\left[\dfrac{1}{\Delta^2}+\dfrac{1}{\Delta^2}\right]^{\alpha/2}\\
=&\dfrac{2^{\alpha/2}}{\Delta^\alpha}R_\alpha(G)
\end{align*}
Similarly, we obtain the upper bound. Moreover, it is easy to see that the equality holds if and only if $G$ is regular. This completes the proof.
\end{proof}
As a consequence, letting $\alpha=1$ we get a relation for Sombor index in terms of second Zagreb index as reported in \cite{17}.
\begin{corollary} Let $G$ be a graph on $n$ vertices and $m$ edges. Then $$ \dfrac{\sqrt{2}}{\Delta}M_2(G)\leq SO(G) \leq \dfrac{\sqrt{2}}{\delta}M_2(G)$$ where the equality holds if and only if $G$ is regular.
\end{corollary}
\subsection{General Sombor index and general sum-connectivity index} 
Lastly, we present a relation between general Sombor index and general sum-connectivity index.
\begin{theorem} Let $G$ be a graph on $n$ vertices and $m$ edges. Then $$ \dfrac{\chi_\alpha(G)}{2^{\alpha/2}}\leq SO_\alpha(G) \leq  \sqrt{m \Delta^\alpha  \chi_\alpha(G)}$$ where the left equality holds if and only if $G$ is regular.
\end{theorem}
\begin{proof} Let $V(G)=\{v_1,v_2,\dots,v_n\}$. Let, for simplicity, $d_i=d_G(v_i) $ for a vertex $v_i$ of $G$. Letting $a_k\to [{d_i}^2+{d_j}^2]^{\alpha/2}$ and $b_k \to [d_i+d_j]^\alpha$ in Lemma \ref{lemma2} for $p=1$ and the sums running over the edges in $G$, we have
\begin{align}
\displaystyle\sum_{v_iv_j\in E(G)}\left[\dfrac{{d_i}^2+{d_j}^2}{d_i+d_j}\right]^\alpha \geq \dfrac{\left[\displaystyle\sum_{v_iv_j\in E(G)}[{d_i}^2+{d_j}^2]^{\alpha/2}\right]^2} {\displaystyle\sum_{v_iv_j\in E(G)}[d_i+d_j]^\alpha} \label{eqn:eq4}
\end{align}
Notice that $\dfrac{{d_i}^2+{d_j}^2}{d_i+d_j}\leq \Delta$ for every edge $v_iv_j$ of $G$. Thus (\ref{eqn:eq4}) becomes
 $$ \dfrac{{SO_\alpha(G)}^2}{\chi_\alpha(G)}\leq m\Delta^\alpha \implies SO_\alpha(G) \leq \sqrt{m \Delta^\alpha  \chi_\alpha(G)}.$$
 For the lower bound, notice that for any edge $v_iv_j$ of $G$, we have ${d_i}^2+{d_j}^2\geq \dfrac{(d_i+d_j)^2}{2}$. Thus
 $$SO_\alpha(G)=\displaystyle\sum_{v_iv_j\in E(G)} [{d_i}^2+{d_j}^2]^{\alpha/2}\geq \displaystyle\sum_{v_iv_j\in E(G)} \dfrac{[d_i+d_j]^{\alpha}}{2^{\alpha/2}}=\dfrac{\chi_\alpha(G)}{2^{\alpha/2}}.$$
\end{proof}
As a consequence, letting $\alpha=1$ we get the lower bound for Sombor index as reported in \cite{17} and the upper bound for Sombor index as found by the present authors in \cite3.
\begin{corollary} Let $G$ be a graph on $n$ vertices and $m$ edges. Then $$ \dfrac{M_1(G)}{\sqrt{2}}\leq SO(G) \leq  \sqrt{m \Delta  M_1(G)}$$ where the left equality holds if and only if $G$ is regular.
\end{corollary}
\section*{Acknowledgments}
The second author is supported by the MATRICS project funded by DST-SERB, Government of India under Grant no. MTR/2017/000403 dated 06/06/2018.

\end{document}